\documentclass[fleqn,10pt]{article}
\RequirePackage[left=28mm,right=28mm,top=35mm,bottom=30mm]{geometry}
\usepackage[utf8]{inputenc}
\usepackage[numbers]{natbib}
\usepackage{hyperref}
\usepackage[english]{babel}
\usepackage[fleqn]{amsmath}
\usepackage{amssymb}
\usepackage{amsthm}
\usepackage[shortlabels]{enumitem} %to index the enumeration
\usepackage{afterpage}
\usepackage{graphicx} %to improve display on 3.2

\theoremstyle{definition}
\newtheorem{schur algebra}[subsection]{Definition}
\newtheorem{schur-weyl}[subsection]{Definition}
\newtheorem{strong epimorphism}[subsection]{Definition}
\newtheorem{weight}[subsection]{Definition}
	\theoremstyle{plain}
	\newtheorem{integraldomainspsi}[subsection]{Example}
\newtheorem{strong epimorphism properties}[subsection]{Proposition}
\newtheorem{eij action}[subsection]{Lemma}
\newtheorem{chi action}[subsection]{Corollary}
\newtheorem{bryant}[subsection]{Theorem}
\newtheorem{Schur-weyl duality iff rho}[subsection]{Corollary}
\newtheorem{surjectivity of rho}[subsection]{Theorem}
\newtheorem{schur weyl duality for finite fields can fail}[subsection]{Theorem}
\newtheorem{schur weyl duality integers}[subsection]{Corollary}
\theoremstyle{remark}
\newtheorem{exemplos}[subsection]{Remarks}
\newtheorem{remark about finite fields}[subsection]{Remarks}
\newtheorem{remark about integers}[subsection]{Remarks}
\title{\vspace*{-3cm}\textbf{{\footnotesize [Note: This is an Accepted Manuscript of an article published by Taylor \& Francis in Communications in Algebra on 13 Nov 2018, available online: http://www.tandfonline.com/10.1080/00927872.2018.1513010 .]}}\\[2cm]
\LARGE	Schur--Weyl duality over commutative rings}
\author{ Tiago Cruz} %put also affiliation
\date{}

\newcommand{\Address}{{
		\bigskip
		\footnotesize
		
		TIAGO CRUZ,\par \textsc{Institute of Algebra and Number Theory}\par \textsc{University of Stuttgart,}\par \textsc{Pfaffenwaldring 57, 70569 Stuttgart, Germany,}\par\nopagebreak
		\textit{E-mail address}, T.~Cruz: \texttt{tiago.cruz@mathematik.uni-stuttgart.de}}}
		
		\allowdisplaybreaks
\begin{document}

\maketitle

\begin{abstract}
The classical case of Schur--Weyl duality states that the actions of the group algebras of $GL_n$ and $S_d$ on the $d^{th}$-tensor power of a free module of finite rank centralize each other. We show that Schur--Weyl duality holds for commutative rings where enough scalars can be chosen whose non-zero differences are invertible. This implies all the known cases of Schur--Weyl duality so far. We also show that Schur--Weyl duality fails for $\mathbb{Z}$ and for any finite field when $d$ is sufficiently large.
\end{abstract}

\section{Introduction}
 Schur--Weyl duality is a connection between the general linear group and the symmetric group.

 More specifically, consider $n, d\in \mathbb{N}$ and let $V=R^n$ be the free module of rank $n$ over a commutative ring with identity $R$. 
 
 The symmetric group $S_d$ acts on the $d^{th}$-tensor power, $V^{\otimes^d}=V\otimes_R\cdots\otimes_R V$, of the module $V$ by place permutation, that is,
 $$\sigma(v_1\otimes\cdots\otimes v_d)=v_{\sigma^{-1}(1)}\otimes\cdots\otimes v_{\sigma^{-1}(d)}, \ \sigma\in S_d, \ v_i\in V.$$ 
	\begin{schur algebra} \citep{green}
	The subalgebra $End_{RS_d}\left( V^{\otimes ^d}\right)$ of the endomorphism algebra $End_R\left( V^{\otimes^d}\right)$ is called the \emph{Schur algebra}. We will denote it by $S_R(n, d)$.
\end{schur algebra}
On the other hand, the general linear group acts on $V$ by multiplication, and thus on the tensor product $V^{\otimes^d}$ by the diagonal action, that is,
$$g(v_1\otimes\cdots\otimes v_d)=gv_1\otimes\cdots\otimes g v_d, \ g\in GL_n(R), \ v_i\in V.$$
These two actions commute, so, by extending these actions to the group algebras, one gets two natural homomorphisms:
$$\rho\colon RGL_n(R)\rightarrow S_R(n, d), \ \psi\colon RS_d\rightarrow End_{RGL_n(R)}\left( V^{\otimes^d}\right). $$

\begin{schur-weyl}
	We say that \emph{Schur--Weyl duality} holds if the two algebra homomorphisms $\rho$ and $\psi$ are surjective.
\end{schur-weyl}
In other words, the image of each homomorphism is the centralizer algebra  for the other action.

In case $R=\mathbb{K}$ is an infinite field, Green, De Concini and Procesi proved that Schur--Weyl duality holds (see \citep{green, de1976characteristic}). Another approach assuming only that Schur--Weyl duality holds for $\mathbb{C}$, which is due to Schur, can be found in \citep{doty2009schur}.  Benson and Doty showed in \citep{benson2009schur} that  Schur--Weyl duality holds for finite fields with order strictly larger than $d$.

When Schur--Weyl duality holds then the category of $S_R(n,d)$-modules is equivalent to the category of homogeneous polynomial representation of degree $d$ of $GL_n(R)$  (see \citep{krause2013polynomial}). This means, that if Schur--Weyl duality holds, one can replace the group algebra $RGL_n(R)$ in the study of homogeneous polynomial representations of degree $d$ of $GL_n(R)$ by the Schur algebra $S_R(n, d)$.

The aim of the present paper is to study Schur--Weyl duality when the ground ring is any commutative ring. We give a sufficient condition for Schur--Weyl duality to hold. This criterion when applied to fields is exactly the one Benson and Doty obtained in \citep{benson2009schur}.

For instance we have,

\begin{surjectivity of rho}
	Let $R$ be a commutative ring with identity.
If $R$ contains a set with more than $d$ elements whose non-zero differences are invertible then Schur--Weyl duality holds for the $d^{th}$-tensor power of a free module with finite rank over $R$.
\end{surjectivity of rho}

In case $R$ is a finite field, we also present a formula involving the vector space dimension and the parameter $d$, to show failure of Schur--Weyl duality in certain cases. This result is contained in Theorem \ref{schur weyl duality for finite fields can fail}.

Theorem \ref{schur weyl duality all cases} and Theorem \ref{schur weyl duality for finite fields can fail} are the main contributions of this paper. 

\section{Some observations on strong epimorphisms}
 First, let us recall the definition of strong epimorphism introduced by Krause in \citep{krause2013polynomial}.
\begin{strong epimorphism}\label{strong epimorphism}
	
	Let $\phi\colon A\rightarrow B$ be an $R$-algebra homomorphism. We say that $\phi$ is a \emph{strong epimorphism} if
	\begin{enumerate}[(i)]
		\item $\phi$ is an epimorphism of $R$-algebras in the categorical sense, that is: For any pair of $R$-algebra homomorphisms $\psi_1, \psi_2\colon B\rightarrow C$, if $\psi_1\circ \phi=\psi_2\circ \phi$, then $\psi_1=\psi_2$.
		\item Let $r\colon A\rightarrow End_R(M)$ be a representation. If there exists an $R$-linear map $h\colon B\rightarrow End_R(M)$ such that $r=h\circ \phi$ then there exists a representation $s\colon B\rightarrow End_R(M)$ such that $r=s\circ \phi$. 
	\end{enumerate}
\end{strong epimorphism}

As noticed in \citep{krause2013polynomial}, the inclusion $\mathbb{Z}\hookrightarrow \mathbb{Q}$ is a strong epimorphism. But clearly it is not surjective. Therefore, one should be aware that strong epimorphism introduced by H. Krause is distinct from strong epimorphism notion in the categorical sense.

 From now on, only the notion defined in (\ref{strong epimorphism}) will be used.

 Strong epimorphism is a stronger notion than epimorphism: 

In fact, for any commutative ring $R$, consider the natural monomorphism $\phi\colon R[x]\rightarrow R\mathbb{Z}$. As $\phi(x)=v_1$ and $v_iv_j=v_{i+j}$, where $v_i$ denotes the basis element $i\in \mathbb{Z}$ of the group algebra $R\mathbb{Z}$, one gets that $\phi$ is an epimorphism. 
Let $r\colon R[x]\rightarrow End_R\left(R[x] \right)$ be the regular representation. As $r(x)$ does not have an inverse, there is not a representation $h\colon R\mathbb{Z}\rightarrow End_R\left(R[x] \right)$ such that $h\circ \phi=r$. Yet, we can find an $R$-linear map $h\colon R\mathbb{Z}\rightarrow End_R\left(R[x] \right)$ defined by $h(v_i)=r(x^{|i|}), \ i\in \mathbb{Z}$. Thus, $\phi$ cannot be a strong epimorphism.

Here are some additional properties of strong epimorphisms.

\begin{strong epimorphism properties}
	Let $f\colon B\rightarrow C$ and $g\colon A\rightarrow B$ be an $R$-algebra homomorphisms.
	\begin{enumerate}[(i)]
		\item If $g$ is surjective and $f$ is a strong epimorphism then $f\circ g$ is a strong epimorphism.
		\item If $g$ is a strong epimorphism and $f$ is an isomorphism then $f\circ g$ is a strong epimorphism.
		\item If $g$ is an epimorphism and $f\circ g$ is a strong epimorphism then $f$ is a strong epimorphism.
	\end{enumerate}
\end{strong epimorphism properties}
\begin{proof}
	Suppose that $g$ is surjective and $f$ is a strong epimorphism. It is clear that $f\circ g$ is an epimorphism. 
	
	Let $r\colon A\rightarrow End_R(M)$ be a representation. Assume that $s\colon C\rightarrow End_R(M)$ is an $R$-linear map such that $s\circ f\circ g=r$. As $g$ is surjective, $s\circ f$ is a representation. As $f$ is a strong epimorphism there is a representation $h\colon C\rightarrow End_R(M)$ such that $h\circ f=s\circ f$. Moreover $h=s$.
	
	Now suppose that $g$ is a strong epimorphism and $f$ is an isomorphism. Let $r\colon A\rightarrow End_R(M)$, $s\colon C\rightarrow End_R(M)$ be a representation and an $R$-linear map, respectively, such that $s\circ f\circ g=r$. As $g$ is a strong epimorphism, there is a representation, $t\colon B\rightarrow End_R(M)$ such that $t\circ g=r$. Considering the representation $t\circ f^{-1}$, $(ii)$ follows.
	
	Finally, assume that $f\circ g$ is a strong epimorphism and $g$ is an epimorphism. It is also clear that $f$ is an epimorphism. Let $t\colon B\rightarrow End_R(M)$ be a representation and let $s\colon C\rightarrow End_R(M)$ be an $R$-linear map such that $s\circ f=t$. Thus, $s\circ f\circ g=t\circ g$ is a representation. By hypothesis, there is a representation $p\colon C\rightarrow End_R(M)$ such that $p\circ f\circ g=t\circ g$. As $g$ is epimorphism then $(iii)$ follows.
\end{proof}

\section{Generalization of Schur--Weyl duality to commutative rings}
In this section, we aim to extend the work of Benson and Doty to commutative rings with the same approach as for finite fields. Moreover, Corollary 4.4 of \citep{benson2009schur}, which says that Schur--Weyl duality holds over finite fields sufficiently large, is crucial to our aim. 

Consider the Lie group $GL_n(\mathbb{C})$ of invertible matrices over $\mathbb{C}$ and $\mathfrak{gl}_n(\mathbb{C})$ its Lie algebra. Let $U$ be the enveloping algebra of the Lie algebra $\mathfrak{gl}_n(\mathbb{C})$.

Thus, $U$ is the associative $\mathbb{C}$-algebra with generators $e_{i, j}, \ 1\leq i, j\leq n$, which satisfy the relation
\begin{align*}
e_{i, j}e_{a, b}-e_{a, b}e_{i, j}=\delta_{a, j}e_{i, b}-\delta_{i, b}e_{a, j}, \quad i, j, a, b=1, \ldots, n.
\end{align*}
Let $V$ be a finite dimensional complex vector space with basis $\{v_1, \ldots, v_n\}$.
Let $\mathfrak{gl}(V)$ be the Lie algebra whose underlying vector space is $End_{\mathbb{C}}(V)$ together with the Lie bracket $[f, g]:=f\circ g-g\circ f$,\linebreak $f,g\in End_{\mathbb{C}}(V)$. 

Consider the representation of the Lie algebra $\mathfrak{gl}_n(\mathbb{C})$ on $V$, $r\colon \mathfrak{gl}_n(\mathbb{C})\rightarrow \mathfrak{gl}(V)$, given by \linebreak $r(e_{i, j})(v_k)=\delta_{j, k}v_i$, $1\leq i, j, k\leq n$.
Thus we have a representation of $\mathfrak{gl}_n(\mathbb{C})$ on $V^{\otimes^d}$, $r\otimes\cdots\otimes r$, given by 
 \begin{align*} \hspace*{-0.9cm}
\left( r\otimes\cdots\otimes r\right) (e_{i, j})(v_{k_1}\otimes\cdots\otimes v_{k_d})=\delta_{j, k_1}v_i\otimes\cdots\otimes v_{k_d}+\cdots+\delta_{j, k_d}v_{k_1}\otimes\cdots\otimes v_{i}, \: 1\leq i, j, k_1, \ldots, k_d\leq n.
\end{align*}
As the category of all representations of $\mathfrak{gl}_n(\mathbb{C})$ is equivalent to the abelian category of all left modules over $U$, one gets $V^{\otimes^d}$ as $U$-module with
\begin{align*}
e_{i, j}(v_{k_1}\otimes\cdots\otimes v_{k_d})=\delta_{j, k_1}v_i\otimes\cdots\otimes v_{k_d}+\cdots+\delta_{j, k_d}v_{k_1}\otimes\cdots\otimes v_{i}, \quad 1\leq i, j, k_1, \ldots, k_d\leq n.
\end{align*}

Let $U_{\mathbb{Z}}'$ be the subring of $U$ generated by the elements $\dfrac{e_{i,j}^m}{m!}, \ 1\leq i\neq j\leq n, \ m\geq 0$ and let $U_{\mathbb{Z}}$ be the subring of $U$ generated by $U_{\mathbb{Z}}'$ and by the elements $\displaystyle\binom{e_{i,i}}{m}:=\dfrac{e_{i,i}(e_{i,i}-1_U)\cdots(e_{i, i}-(m-1)1_U)}{m!}$, $1\leq i\leq n$, $m\geq 0$.

We shall see how these elements act on $V^{\otimes^d}$. To see that, it is useful to recall the notion of weight of a simple tensor.

\begin{weight}\citep{benson2009schur}
	The weight of a simple tensor $v_{j_1}\otimes\cdots\otimes v_{j_d}$, denoted by $\omega(v_{j_1}\otimes\cdots\otimes v_{j_d})$, is the composition $(\lambda_1, \ldots, \lambda_n)$ of $d$ in at most $n$ parts where $\lambda_j=|\{1\leq \mu\leq n\colon i_{\mu}=j \}|$, $j=1, \ldots, n$.
\end{weight}

In the next lemma, the action of these elements on $V^{\otimes^d}$ will be computed explicitly, as the computation in the proof of Lemma 4.1 of \citep{benson2009schur} has a minor mistake but the result still holds. In there, the usual algebra action was used instead of the universal enveloping algebra action associated with the usual Lie algebra action.

\begin{eij action}\label{eij action}
	Let $v_{j_1}\otimes\cdots\otimes v_{j_d}$ be any basis element of $V^{\otimes^d}$. Consider $\lambda=\omega(v_{j_1}\otimes\cdots\otimes v_{j_d})$ and let $\{\varepsilon_1, \ldots, \varepsilon_n\}$ be the canonical basis of $\mathbb{Z}^n$.
	
	Then, for any $m\geq 0$ and $1\leq i\neq j\leq n$, 
	\begin{align*}
	\dfrac{e_{i, j}^m}{m!}(v_{j_1}\otimes\cdots\otimes v_{j_d})&=\begin{cases} 
	 \text{sum of } \binom{\lambda_j}{m}  \text{ distinct simple tensors} \\ \text{written as } v_{k_1}\otimes\cdots\otimes v_{k_d}, \ \footnotesize{ k_l\in \{j_l, i\}}, \quad & \text{if } \lambda_j\geq m\\ \scriptstyle{1\leq l \leq d},\textstyle  \text{ with weight } \lambda+m\varepsilon_i-m\varepsilon_j 
	\\[2.5mm] \hfil  0,& \text{otherwise} 
	\end{cases};\\
	\binom{e_{i, i}}{m}(v_{j_1}\otimes\cdots\otimes v_{j_d})&=\binom{\delta_{i, j_1}+\cdots+\delta_{i, j_d}}{m}(v_{j_1}\otimes\cdots\otimes v_{j_d})=\binom{\lambda_i}{m}(v_{j_1}\otimes\cdots\otimes v_{j_d}).
	\end{align*}
\end{eij action}
\begin{proof}
	Fix $1\leq i\neq j\leq n$. We will prove the formulas by induction on $m$.
	
	Consider $m=1$. If $\lambda_j=0$ then $\delta_{j, j_1}=\ldots=\delta_{j, j_d}=0$, therefore $e_{i,j}(v_{j_1}\otimes\cdots\otimes v_{j_d})=0$. If $\lambda_{j}\geq 1$ then there exist $\lambda_j$ numbers $a\in [1, d]$ such that $\delta_{j, j_a}=1$, that is, $e_{i,j}(v_{j_1}\otimes\cdots\otimes v_{j_d})$ is the sum of $\lambda_j$ simple tensors $v_{k_1}\otimes\cdots\otimes v_{k_d}$, $k_l\in\{j_l, i\}$, $1\leq l\leq d$, with weight $\lambda+\varepsilon_i-\varepsilon_j$.
	\begin{align*}
	\binom{e_{i,i}}{1}(v_{j_1}\otimes\cdots\otimes v_{j_d})
	&=e_{i, i}(v_{j_1}\otimes\cdots\otimes v_{j_d})= \delta_{i, j_1}v_{i}\otimes\cdots\otimes v_{j_d}+\cdots+\delta_{i, j_d}v_{j_1}\otimes\cdots\otimes v_i 
	\\&=(\delta_{i, j_1}+\cdots+\delta_{i, j_d})(v_{j_1}\otimes\cdots\otimes v_{j_d}).
	\end{align*}
	Suppose now that $m>1$ and the results holds for $m-1$. 
	
	If $\lambda_j<m-1$ then $\dfrac{e_{i, j}^m}{m!}(v_{j_1}\otimes\cdots\otimes v_{j_d})=\dfrac{e_{i, j}}{m}(0)=0$, by the induction hypothesis.
	
	If $\lambda_j\geq m-1$ then \begin{align*}
	\dfrac{e_{i, j}^m}{m!}(v_{j_1}\otimes\cdots\otimes v_{j_d})&=\dfrac{e_{i, j}}{m}\left(\sum_{\substack{k_l\in \{j_l, \ i \}, \ 1\leq l\leq d\\ \omega(v_{k_1}\otimes\cdots\otimes v_{k_d})=\lambda+(m-1)\varepsilon_i-(m-1)\varepsilon_j}}v_{k_1}\otimes\cdots\otimes v_{k_d} \right) \\&=\frac{1}{m}\sum_{\substack{k_l\in \{j_l, \ i \}, \ 1\leq l\leq d\\ \omega(v_{k_1}\otimes\cdots\otimes v_{k_d})=\lambda+(m-1)\varepsilon_i-(m-1)\varepsilon_j}}e_{i, j}(v_{k_1}\otimes\cdots\otimes v_{k_d}).
	\end{align*}
	If $\lambda_j= m-1$ then  $\delta_{k_l, j}=0, $ $1\leq l\leq d$, thus $\dfrac{e_{i, j}^m}{m!}(v_{j_1}\otimes\cdots\otimes v_{j_d})=0$.
	
	Suppose now $\lambda_j\geq m$. By the previous computation, $\dfrac{e_{i, j}^m}{m!}(v_{j_1}\otimes\cdots\otimes v_{j_d})$ is the sum of $\binom{\lambda_j}{m-1}$ elements written as $e_{i, j}(v_{k_1}\otimes\cdots\otimes v_{k_d})$. Each element is the sum of $\lambda_j-(m-1)$ simple tensors with weight $\lambda+m\varepsilon_i-m\varepsilon_j$. As a result, $\dfrac{e_{i, j}^m}{(m-1)!}(v_{j_1}\otimes\cdots\otimes v_{j_d})$ is the sum of $\binom{\lambda_j}{m-1}(\lambda_j-(m-1))$ simple tensors with weight $\lambda+m\varepsilon_i-m\varepsilon_j$. However, in $\dfrac{e_{i, j}^m}{(m-1)!}(v_{j_1}\otimes\cdots\otimes v_{j_d})$ each simple tensor has multiplicity $m$.
	So, $\dfrac{e_{i, j}^m}{m!}(v_{j_1}\otimes\cdots\otimes v_{j_d})$ is the sum of $\binom{\lambda_j}{m}$ distinct simple tensors with weight $\lambda+m\varepsilon_i-m\varepsilon_j$. 
	
	It remains to show the inductive step for $\displaystyle\binom{e_{i, i}}{m}$:
	\begin{align*}
	\binom{e_{i,i}}{m}(v_{j_1}\otimes\cdots\otimes v_{j_d})&=\dfrac{e_{i,i}(e_{i,i}-1_U)\cdots(e_{i, i}+(-m+1)1_U)}{m!}(v_{j_1}\otimes\cdots\otimes v_{j_d})
	\\&= \dfrac{e_{i,i}(e_{i,i}-1_U)\cdots(e_{i, i}+(-m+2)1_U)}{(m-1)!m}(e_{i, i}+(-m+1)1_U)(v_{j_1}\otimes\cdots\otimes v_{j_d}) \\&=\frac{1}{m}\binom{e_{i, i}}{m-1}\left( \delta_{i, j_1}+\cdots+\delta_{i, j_d}-(m-1)\right) (v_{j_1}\otimes \cdots\otimes v_{j_d})
	\\&=\dfrac{\delta_{i, j_1}+\cdots+\delta_{i, j_d}-(m-1)}{m}\binom{e_{i, i}}{m-1}(v_{j_1}\otimes \cdots\otimes v_{j_d})
	\\&=\dfrac{\delta_{i, j_1}+\cdots+\delta_{i, j_d}-(m-1)}{m}\binom{\delta_{i, j_1}+\cdots+\delta_{i, j_d}}{m-1}(v_{j_1}\otimes \cdots\otimes v_{j_d}) 
	\\&= \binom{\delta_{i, j_1}+\cdots+\delta_{i, j_d}}{m}(v_{j_1}\otimes \cdots\otimes v_{j_d}). \tag*{\qedhere}\end{align*}
\end{proof}
By direct computation, we see that the action of $U$ on $V^{\otimes^d}$ commutes with the action of $S_d$ on $V^{\otimes^d}$.
Let $V_{\mathbb{Z}}$ be the free $\mathbb{Z}$-module with basis $\{v_1, \ldots, v_n\}$. By  Lemma \ref{eij action}, $\left( V_{\mathbb{Z}}\right)^{\otimes^d}$ is an  $U_{\mathbb{Z}}$-module.

Therefore, for any commutative ring with identity, $R\otimes_{\mathbb{Z}}V_{\mathbb{Z}}^{\otimes^d}$ has the structure of $R\otimes_{\mathbb{Z}}U_{\mathbb{Z}}$-module. Set $U_R=R\otimes_{\mathbb{Z}}U_{\mathbb{Z}}$,  $U_R'=R\otimes_{\mathbb{Z}}U_{\mathbb{Z}}'$ and $V_R=R\otimes_{\mathbb{Z}} V_{\mathbb{Z}}$. So, one gets the representation $\chi\colon U_R\rightarrow End_R\left( V_R^{\otimes^d}\right)$ associated with the module $V_R^{\otimes^d}$.
As the actions of $U_R$ and $S_d$ on $V_R^{\otimes^d}$ commute, one obtains the algebra homomorphisms $$\chi_R\colon U_R\rightarrow End_{RS_d}\left( V_R^{\otimes^d}\right), \   \chi'_R\colon U_R'\rightarrow End_{RS_d}\left( V_R^{\otimes^d}\right), $$ with $\chi'_R:=(\chi_R)_{|_{U_R'}}$.  

\begin{chi action}\citep[Lemma 4.1]{benson2009schur}
	For any $m>d$, $\chi_R\!\left(\! 1\otimes\dfrac{e_{i, j}^m}{m!}\!\right) \!=\chi_R\!\left( 1\otimes\displaystyle\binom{e_{i,i}}{m}\right) \!=0$, $ 1\leq i\neq j\leq n$. 
\end{chi action}
Using the basis $\{1\otimes v_1, \ldots, 1\otimes v_n\}$ of $V_R$ one gets, as analogue of $\rho$ and $\psi$, the algebra homomorphisms $$\rho_{R}\colon RGL(V_{R})\rightarrow End_{RS_d}\left( V_{R}^{\otimes^d}\right), \ \psi_R\colon RS_d\rightarrow End_{U_R}\left( V_{R}^{\otimes^d}\right).$$
It is clear that $\rho_R$ is surjective if and only if $\rho\colon RGL_n(R)\rightarrow S_R(n, d)$ is surjective. So, we will focus on $\rho_R$.

For finite fields $R=\mathbb{K}$ with $|\mathbb{K}|>d$, Benson and Doty showed that $\rho_{\mathbb{K}}$ is surjective \citep[Theorem 4.3]{benson2009schur}. By Remark 4.6 of their article, for any commutative ring $R$ the maps
\begin{align}
\chi_R'\colon U_R'\rightarrow End_{RS_d}\left( V_R^{\otimes^d}\right), \ \psi_R\colon RS_d\rightarrow End_{U_R}\left( V_R^{\otimes^d}\right) \label{results from bd}
\end{align} are surjective for any $n, d\in \mathbb{N}$.

We note the following improvement of a version of Schur--Weyl duality \citep[Lemma 2.4]{bryant2009lie}.

\begin{bryant}
	Let $R$ be a commutative ring with identity. Then the algebra homomorphism \linebreak$\psi\colon RS_d\rightarrow End_{S_R(n, d)}\left( (R^n)^{\otimes^d}\right)$ is surjective for any $n, d\in \mathbb{N}$.
\end{bryant}
\begin{proof} Suppose $R$ a commutative ring with identity.
	By (\ref{results from bd}),
	\begin{align*}
	\psi_R(RS_d)=End_{U_R}\left(V_R^{\otimes^d} \right) =End_{\chi_R(U_R)}\left(V_R^{\otimes^d} \right) =End_{End_{RS_d}\left( V_R^{\otimes^d}\right) }\left( V_R^{\otimes^d}\right).
	\end{align*}
	Since $V_R^{\otimes^d}\cong (R^n)^{\otimes^d}$ as $RS_d$-modules, in the canonical way, it follows \begin{equation*}
	\psi(RS_d)= End_{S_R(n, d)}\left( (R^n)^{\otimes^d}\right). \tag*{\qedhere}
	\end{equation*}
\end{proof}

As a result, the argument given in \citep[Corollary 4.4]{benson2009schur} still holds for commutative rings.
\begin{Schur-weyl duality iff rho}
	Let $R$ be a commutative ring with identity. Then Schur--Weyl duality holds if and only if the algebra homomorphism $\rho\colon RGL_n(R)\rightarrow S_R(n, d)$ is surjective.
\end{Schur-weyl duality iff rho}
\begin{proof}
	Let $R$ be a commutative ring with identity.
	
	One implication is clear.
	
	Now suppose that $\rho\colon RGL_n(R)\rightarrow S_R(n, d)$ is a surjective map. Then,
	\begin{align*}
	\psi(RS_d)=End_{S_R(n, d)}\left( (R^n)^{\otimes^d}\right)=End_{\rho(RGL_n(R))}\left( (R^n)^{\otimes^d}\right)=End_{RGL_n(R)}\left( (R^n)^{\otimes^d}\right).
	\end{align*}
	Thus, $\psi$ is surjective and Schur--Weyl duality holds.
\end{proof}
Hence, the study of Schur--Weyl duality can be reduced to studying the surjectivity of $\rho$.
 When $R=\mathbb{K}$ is a field this result was already known. Following the work of Krause \citep{krause2013polynomial}, one sees that as a consequence of this result, Schur--Weyl duality holds if and only if the category of homogeneous polynomial representations of degree $d$ of $GL_n(\mathbb{K})$ is equivalent to the category of modules over the Schur algebra $S_{\mathbb{K}}(n, d)$.

Let $R^*$ be the set of all units of the commutative ring $R$. Then we have the following result.

\begin{surjectivity of rho}\label{schur weyl duality all cases}
	Let $n, d\in \mathbb{N}$ be natural numbers.
	Let $R$ be a commutative ring with identity that contains a set $S$ which has the following properties:
	\begin{enumerate}
		\item $\forall x, y\in S, \ x\neq y\implies x-y\in R^*$;
		\item $|S|>d$.
	\end{enumerate}
Then the algebra homomorphism $\rho\colon RGL_n(R)\rightarrow S_R(n, d)$ is surjective, that is, Schur--Weyl duality holds.
\end{surjectivity of rho}
\begin{proof}
	Suppose $R$ a commutative ring satisfying the above conditions.
	
	For any $t\in R$, $1\leq i, j\leq n$, set $E_{i, j}(t)=\mathrm{id}_{V_{R}}+t\chi_{R}(1\otimes e_{i, j})$.
	
	Note that a matrix $A\in GL_n(R)$ if and only if its determinant $\det(A)$ is a unit. So as \linebreak$\det[E_{i, j}(t)]_{\{1\otimes v_1, \ldots, 1\otimes v_n\}}=1$, one obtains that $E_{i, j}(t)\in GL(V_R), \ 1\leq i\neq j\leq n, \ t\in R$.
	
	The formula, used in \citep[Lemma 4.2 (6)]{benson2009schur} with $R=\mathbb{K}$ a field, \begin{align}
	\rho_{R}\left( E_{i, j}(t)\right)=\sum_{m=0}^d t^m \chi_{R}'\left( 1\otimes \dfrac{e_{i, j}^m}{m!}\right), \ 1\leq i\neq j\leq n, \ t\in R. \label{formula}
	\end{align} still holds for any commutative ring.
	
	By hypothesis, there are elements $t_0, \ldots, t_d\in S\subset R$ such that $t_q-t_p\in R^{*}$, $0\leq q<p\leq d$.
	Applying (\ref{formula}), one obtains
	\begin{align*}
\rho_{R}\left( E_{i, j}(t_k)\right)=\sum_{m=0}^d t_k^m \chi_{R}'\left( 1\otimes \dfrac{e_{i, j}^m}{m!}\right), \ 0\leq k\leq d, \ 1\leq i\neq j\leq n.
	\end{align*} 
	This system is represented by the matrix equation
	\begin{align*}
	\begin{bmatrix}
	1 & t_0 & \cdots & t_0^d\\ \vdots & \vdots & \vdots & \vdots \\ 1 & t_d & \cdots & t_d^d
	\end{bmatrix}\begin{bmatrix}
	\mathrm{id}_{V_{R}^{\otimes^d}} \\ \chi'_{R}\left( 1\otimes e_{i, j}\right) \\ \vdots \\ \chi'_{R}\left( 1\otimes \dfrac{e_{i, j}^d}{d!}\right) 
	\end{bmatrix}
	=\begin{bmatrix}
	\rho'_{R}(E_{i, j}(t_0))\\\vdots\\ \rho'_{R}(E_{i, j}(t_d))
	\end{bmatrix}.
	\end{align*}
	The matrix $[t_k^l]$ is a Vandermonde matrix, so its determinant is $\underset{0\leq q<p\leq d}{\prod}(t_q-t_p)\in R^*$. Thus, it is invertible, or in other words, there are scalars $\alpha_{m, l}\in R$ such that 
	\begin{align*}
	\chi_{R}'\left( 1\otimes\dfrac{e_{i, j}^m}{m!}\right)=\sum_{l=0}^d \alpha_{m, l}\rho_{R}(E_{i, j}(t_l)), \ 0\leq m\leq d, \ 1\leq i\neq j\leq n. 
	\end{align*}
	
	As $\left\lbrace \left( 1\otimes\dfrac{e_{i, j}^m}{m!} \right) \colon m\geq 0, \ 1\leq i\neq j\leq n \right\rbrace $ is a generator set for $U_R'$ and $\chi_{R}'\left( 1\otimes\dfrac{e_{i, j}^m}{m!}\right)=0$, when $m>d$, it follows that $im \chi'_R\subset im \rho_R$.
	
	By (\ref{results from bd}), $\rho_R$ is surjective.
	Therefore, $\rho$ is surjective.
\end{proof}
%

%\newgeometry{left=28mm,right=28mm,top=29mm,bottom=30mm} %adjust page to put remarks on the top
%\afterpage{\restoregeometry}
\begin{exemplos}\
	\begin{enumerate}[(1)]
		\item For any commutative ring $R$ with identity, Schur--Weyl duality holds for $d=1$. In fact, any commutative ring with identity contains the set $\{0, 1\}$.
		\item When $R=\mathbb{K}$ is a field with more than $d$ elements, we can choose $S=\mathbb{K}$. Therefore, the Theorem \ref{schur weyl duality all cases} contains all the known cases of the classical Schur--Weyl duality so far.
		\item Let $\mathbb{K}$ be a field with more than $d$ elements. Consider $R=\mathbb{K}[x_1,\ldots, x_k]$ the polynomial ring and fix an arbitrary natural $n$. We can apply the Theorem \ref{schur weyl duality all cases} with $S=\mathbb{K}$. Thus, Schur--Weyl duality holds for the polynomial ring $R$.
	\end{enumerate}
\end{exemplos}

\section{Some cases when Schur--Weyl duality fails}
In this final section, the aim is to present some situations where Schur--Weyl duality does not hold. We will show that the map $\psi\colon RS_d\rightarrow End_{RGL_n(R)}\left( (R^n)^{\otimes^d}\right)$ can be surjective in cases where Schur--Weyl duality fails. First, studying the map $\rho$, we can extend Theorem 5.1 of \citep{benson2009schur} to find situations where Schur--Weyl duality fails.

\begin{schur weyl duality for finite fields can fail}\label{schur weyl duality for finite fields can fail}
	Let $\mathbb{K}$ be a finite field and fix $n\in \mathbb{N}$. For $d$ sufficiently large,  Schur--Weyl duality fails. 
	
	More precisely,  Schur--Weyl duality fails for all $d$ that satisfy $\binom{n^2+d-1}{d}>\overset{n}{\underset{i=1}{\prod}}\left( |\mathbb{K}|^n-|\mathbb{K}|^{i-1}\right) $.
\end{schur weyl duality for finite fields can fail}
\begin{proof}
	If $\rho$ is surjective then $\dim_{\mathbb{K}}(\mathbb{K}GL_n(\mathbb{K}))\geq \dim_{\mathbb{K}}(S_{\mathbb{K}}(n, d))$. Computing a base for the $S_d$-invariants of $\left( End_{\mathbb{K}}\left( V\right)\right)^{\otimes^d}$, one can show that $\dim S_{\mathbb{K}}(n, d)=\binom{n^2+d-1}{d}$. 
	
	Therefore, if $\binom{n^2+d -1}{d}>|G|=\dim_{\mathbb{K}}\mathbb{K}GL_n(\mathbb{K})\geq \dim_{\mathbb{K}}\rho(\mathbb{K}G)$ then $\rho$ cannot be surjective.
	
	It is clear that $A\in GL_n(\mathbb{K})$ if and only if its columns are linearly independent in $\mathbb{K}^n$. We also note that any $n$-dimensional vector space over a finite field has $|\mathbb{K}|^n$ elements.
	
	The column $i$ must belong to the complement of the vector space generated by the columns indexed by $\{1, \ldots, i-1\}$. So for the column $i$ one has $|\mathbb{K}|^n-|\mathbb{K}|^{i-1}$ choices. So, the order of $GL_n(\mathbb{K})$ is $\overset{n}{\underset{i=1}{\prod}}\left( |\mathbb{K}|^n-|\mathbb{K}|^{i-1}\right) $. And the result follows.
\end{proof}

\begin{remark about finite fields}\
	\begin{enumerate}[(1)]
		\item So we know that for each field $\mathbb{K}$ and $n\in\mathbb{N}$, Schur--Weyl duality holds for $d=1, \ldots, |\mathbb{K}|-1$. It fails for $d\geq d_0$, where $d_0$ is the minimum natural number that satisfies\linebreak $\binom{n^2+d_0-1}{n^2-1}>\overset{n}{\underset{i=1}{\prod}}\left( |\mathbb{K}|^n-|\mathbb{K}|^{i-1}\right)$. It remains unknown what happens to Schur--Weyl duality for $d=|\mathbb{K}|, \ldots, d_0-1$.
		\item Considering $\mathbb{K}=\mathbb{F}_2$ and $n=2$, Schur--Weyl duality holds for $d=1$ and fails for $d\geq 2$.
	\end{enumerate}
\end{remark about finite fields}

\begin{schur weyl duality integers} \label{schur weyl duality integers}
	Let $n$ and $a$ be natural numbers such that $	\binom{n^2+a-1}{n^2-1}>\underset{i=1}{\overset{n}{\prod}}\left( 2^n-2^{i-1}\right) .$
	
	If $d\geq a$ then the homomorphism $\rho\colon\mathbb{Z} GL_n(\mathbb{Z})\rightarrow S_{\mathbb{Z}}(n, d)$ is not surjective.
\end{schur weyl duality integers}
\begin{proof}
	Suppose $n, d$ in the above conditions. The ring $\mathbb{Z}$ is a Euclidean ring, so the elements $E_{i, j}(t)$,\linebreak $1\leq i\neq j\leq n$, $t\in \mathbb{Z}$ and $E_{i, i}(t)$, $1\leq i\leq n$, $t\in \{1, -1\}$ generate the group $GL_n(\mathbb{Z})$ (see for example \citep{MR0207856}, \citep[Section~9.1.3]{serre2002matrices}).

	 For each commutative ring $R$, we will denote by $\alpha\!\colon\! R\otimes_{\mathbb{Z}}\! \left( \mathbb{Z}^n\right)^{\otimes^d}\!\rightarrow\! \left( R^n\right)^{\otimes^d}\!$ and $\beta\!\colon\! R\otimes_{\mathbb{Z}} \!S_{\mathbb{Z}}(n, d)\!\rightarrow\! S_R(n,d)$ the canonical $R$-isomorphisms.
	
	It is clear that $\beta\left( \left( \mathrm{id}_{\mathbb{F}_2}\otimes_{\mathbb{Z}}\rho_{\mathbb{Z}}\right) (1\otimes E_{i, i}(s)) \right) = \rho_{\mathbb{F}_2}\left( I_n\right)  $, for $i=1, \ldots, n$ and $s\in \{1, -1\}.$
	
	Fix $t\in \mathbb{Z}$ and $1\leq i\neq j\leq n$.
	\begin{align*}
	\beta&( ( \mathrm{id}_{\mathbb{F}_2}\otimes\rho_{\mathbb{Z}})  (1\otimes E_{i, j}(t)))(e_{i_1}\otimes\cdots\otimes e_{i_d})
	=\alpha\left( \left( \mathrm{id}_{\mathbb{F}_2}\otimes\rho_{\mathbb{Z}}\right) (1\otimes E_{i, j}(t))(1\otimes e_{i_1}\otimes\cdots\otimes e_{i_d}\right)\\
	&=\alpha\left(1\otimes E_{i, j}(t)(e_{i_1})\otimes\cdots\otimes E_{i, j}(t)(e_{i_d}) \right)
	=(e_{i_1}+\delta_{j, i_1}te_i)\otimes\cdots\otimes (e_{i_d}+ \delta_{j, i_d}te_i)
	\\&=\begin{cases}
	e_{i_1}\otimes\cdots\otimes e_{i_d},& \ \text{ if } t=2k, \ k \in \mathbb{Z}\\
	(e_{i_1}+\delta_{j, i_1}e_i)\otimes\cdots\otimes(e_{i_d}+\delta_{j, i_d}e_i),& \ \text{ if } t=2k+1, \ k\in \mathbb{Z}
	\end{cases}.   
	\end{align*}
	Therefore, $\beta\left( \left( \mathrm{id}_{\mathbb{F}_2}\otimes\rho_{\mathbb{Z}} \right) (1\otimes E_{i, j}(t))\right)=\rho_{\mathbb{F}_2}(I_n)$ or $\beta\left( \left( \mathrm{id}_{\mathbb{F}_2}\otimes\rho_{\mathbb{Z}}\right)  (1\otimes E_{i, j}(t))\right)=\rho_{\mathbb{F}_2}(E_{i, j}(1))$.
	So, \begin{align*}
	\beta(\left( \mathrm{id}_{\mathbb{F}_2}\otimes \rho_{\mathbb{Z}}\right) (\mathbb{F}_2\otimes\mathbb{Z}GL_n(\mathbb{Z})) )\subset \rho_{\mathbb{F}_2}(\mathbb{F}_2GL_n(\mathbb{F}_2)).
	\end{align*}
	Thus, if $\rho_{\mathbb{Z}}$ is surjective then $\mathrm{id}_{\mathbb{F}_2}\otimes \rho_{\mathbb{Z}}$ is surjective. This would imply that $S_{\mathbb{F}_2}(n, d)\subset \rho_{\mathbb{F}_2}(\mathbb{F}_2GL_n(\mathbb{F}_2))$, that is, $\rho_{\mathbb{F}_2}$ is surjective, which is a contradiction.
\end{proof}

\begin{remark about integers}\
	\begin{enumerate}[(1)]
		\item The study of Schur--Weyl duality over the integers when $n=2$ is complete: As we have seen before it holds for $d=1$, but fails for $d\geq 2$.
		\item The argument given in Theorem 2.1 of \citep{benson2009schur} is still true for all commutative rings, so in order to find more situations where Schur--Weyl duality does not hold using the map $\psi$ one must look for situations where $d\geq n$.
		\item As last observation, we can see that the surjectivity of  $\psi\colon RS_d\rightarrow End_{RGL_n(R)}\left( \left( R^n\right)^{\otimes^d}\right)$ is not enough for Schur--Weyl duality to hold. In fact, this follows from \ref{schur weyl duality integers} (when $n=d=2$) and applying $R=\mathbb{Z}$ in the next example.
	\end{enumerate}
\end{remark about integers}

\begin{integraldomainspsi}
	Let $R$ be an integral domain with characteristic different from two. Then the homomorphism $\psi\colon RS_2 \rightarrow End_{RGL_2(R)}((R^2)^{\otimes ^2})$ is surjective.
\end{integraldomainspsi}
\begin{proof}
	Consider $t\in End_{RG}\left( (R^2)^{\otimes^2} \right)$. Let $(e_i)_{i=1, 2}$ be the canonical basis of $R^2$.
	
	Thus, there are coefficients $t_{i_1, i_2}^{j_1, j_2}\in R$ such that
	$$t(e_{j_1}\otimes e_{j_2})=\sum_{i_1, i_2=1}^2 t_{i_1, i_2}^{j_1, j_2} e_{i_1}\otimes e_{i_2}, \quad 1\leq j_1, j_2\leq 2. $$
	
	For each element $g=[g_{i, j}]\in G$, we have $ge_j=\underset{i=1}{\overset{2}{\sum}} g_{i, j}e_i$, $j=1, 2$.
	
	We note that for any pair $1\leq j_1, j_2\leq 2$,
	
	\begin{align*}
	g(t(e_{j_1}\otimes e_{j_2})) &= g\left( \sum_{i_1, i_2=1}^{2} t_{i_1, i_2}^{j_1, j_2} e_{i_1}\otimes e_{i_2} \right) 
	=\sum_{i_1, i_2=1}^{2} \sum_{k_1, k_2=1}^2  t_{i_1, i_2}^{j_1,  j_2}  g_{k_1, i_1} g_{k_2, i_2} e_{k_1}\otimes e_{k_2}
	\intertext{and}
	t(g(e_{j_1}\otimes e_{j_2})) &= t\left( \sum_{i_1, i_2=1}^2 g_{i_1, j_1} g_{i_2, j_2} e_{i_1}\otimes e_{i_2} \right)  = \sum_{i_1, i_2=1}^2 \sum_{k_1, k_2=1}^2 g_{i_1, j_1} g_{i_2, j_2} t_{k_1, k_2}^{i_1, i_2} e_{k_1}\otimes e_{k_2}.
	\end{align*}
	
	Hence,\begin{align}
	\sum_{i_1, i_2=1}^{2}  t_{i_1, i_2}^{j_1, j_2}  g_{k_1, i_1} g_{k_2, i_2} = \sum_{i_1, i_2=1}^2  g_{i_1, j_1} g_{i_2, j_2} t_{k_1, k_2}^{i_1, i_2}, \quad 1\leq k_1,  k_2\leq 2. \label{polinomio carter}
	\end{align}
	
	We claim that $t_{2, 2}^{1, 2}=0$. This is obtained through the following observations:
	
	In equation $(\ref{polinomio carter})$, consider $g=\begin{bmatrix}
	1 & 1\\ 0 & 1
	\end{bmatrix}\in GL_2(R)$ and the following cases:
	\begin{itemize}
		\item Fix $(k_1, k_2)=(2, 2)$ and $(j_1, j_2)=(2, 1)$. The left hand side of $(\ref{polinomio carter})$ is $t_{2, 2}^{2, 1}$, whereas the right hand side is $t_{2, 2}^{2, 1}+t_{2, 2}^{1, 1}$. Then $t_{2, 2}^{1, 1}=0$.\\[-\baselineskip]
		\item Fix $(k_1, k_2)=(2, 2)$ and $(j_1, j_2)=(2, 2)$. The left hand side of $(\ref{polinomio carter})$ is $t_{2, 2}^{2, 2}$, whereas the right hand side is $t_{2, 2}^{1, 2}+ t_{2, 2}^{2, 1}+ t_{2, 2}^{2, 2}$. Then $t_{2, 2}^{1, 2}+ t_{2, 2}^{2, 1}=0$.\\[-\baselineskip]
		\item Fix $(k_1, k_2)=(2, 1)$ and $(j_1, j_2)=(1, 2)$. The left hand side of $(\ref{polinomio carter})$ is $t_{2, 1}^{1, 2}+ t_{2, 2}^{1, 2}$, whereas the right hand side is $t_{2, 1}^{1, 1}+ t_{2, 1}^{1, 2}$. Then $t_{2, 2}^{1, 2}= t_{2, 1}^{1, 1}$.\\[-\baselineskip]
		\item Fix $(k_1, k_2)=(2, 1)$ and $(j_1, j_2)=(2, 1)$. The left hand side of $(\ref{polinomio carter})$ is $t_{2, 1}^{2, 1}+ t_{2, 2}^{2, 1}$, whereas the right hand side is $t_{2, 1}^{1, 1}+ t_{2, 1}^{2, 1}$. Then $t_{2, 2}^{2, 1}= t_{2, 1}^{1, 1}$.
	\end{itemize}
	
	Therefore,
	\begin{align}
	0=t_{2, 2}^{1, 2}+ t_{2, 2}^{2, 1} = t_{2, 2}^{1, 2} + t_{2, 1}^{1, 1} = t_{2, 2}^{1, 2} + t_{2, 2}^{1, 2} = 2 t_{2, 2}^{1, 2}.
	\end{align}
	By assumption, $R$ is an euclidian domain and since $2\neq 0$ it follows that $t_{2, 2}^{1, 2}=0$.
	
	Using the same arguments with the matrix $g=\begin{bmatrix}
	1 & 0\\ 1 & 1
	\end{bmatrix}\in GL_2(R)$, we obtain $t_{1, 1}^{1, 2}=0$. Hence, $t_{i_1, i_2}^{1, 2}=0$ if $(i_1, i_2)\nsim (1, 2)$.
	
	Therefore, defining $ t_{\sigma} :=t_{\sigma^{-1}(1), \sigma^{-1}(2)}^{1, 2}$, we obtain
	\begin{align}
	t(e_1\otimes e_2)= \sum_{\sigma\in S_2} t_{\sigma^{-1}(1), \sigma^{-1}(2)}^{1, 2} e_{\sigma^{-1}(1)}\otimes e_{\sigma^{-1}(2)} = \sum_{\sigma \in S_2} t_{\sigma} \sigma (e_1\otimes e_2)
	\end{align}
	Now, since the value of $t$ on $e_1\otimes e_2$ determines the value of $t$ on any element basis $e_{j_1}\otimes e_{j_2}$ (see proof of \citep[Theorem 2.1.]{benson2009schur}) we conclude that $t=\psi\left( \underset{\sigma\in S_2}{\sum} t_{\sigma}\sigma\right) \in \psi(RS_2)$. So, $\psi$ is surjective.
\end{proof}

\newpage
\section*{Acknowledgments}
All these results are part of my master's thesis \emph{Schur--Weyl duality} written in Portuguese and defended in June 2017 at the University of Coimbra. I would like to thank my master's thesis advisors Doctor Ana Paula Santana and Doctor Ivan Yudin for all the guidance provided and their valuable comments on this work.

\Address
\end{document}